%% file: mmc_popsized.tex
\documentclass{amsart}
\usepackage{amsmath,amsthm}
\usepackage{amsfonts}
\usepackage{amssymb}
\usepackage{color}
\usepackage{graphicx}
\usepackage{subcaption}
\usepackage[comma]{natbib}
\usepackage{url}

\theoremstyle{plain}
\newtheorem{theorem}{Theorem}[section]
\newtheorem{corollary}[theorem]{Corollary}
\newtheorem{proposition}[theorem]{Proposition}
\newtheorem{lemma}[theorem]{Lemma}

\theoremstyle{remark}
\newtheorem{remark}[theorem]{Remark}

\newcommand{\N}{\mathbb{N}}

\title[Multiple-merger coalescents and population size changes]{Cannings models, population size changes and multiple-merger coalescents}
\author{Fabian Freund}
\address{Crop Plant Biodiversity and Breeding Informatics Group (350b), Institute of Plant Breeding, Seed Science and Population Genetics,   University of Hohenheim, Fruwirthstrasse 21, 70599 Stuttgart, Germany}

 \email{fabian.freund@uni-hohenheim.de}
\thanks{Many thanks to G. Achaz and S. Matuszewski for initiating this article by searching for a Cannings model leading to the Beta coalescent with exponential growth, as well as for helpful comments. I want to further thank two anonymous referees for constructive suggestions that improved the readability and correctness of the manuscript, and for pointing out Lemma 5. I was funded by DFG grant FR 3633/2-1
through Priority Program 1590: Probabilistic Structures in Evolution. }
\begin{document}

\begin{abstract}
Multiple-merger coalescents, e.g. $\Lambda$-$n$-coalescents, have been proposed as models of the genealogy of $n$ sampled individuals for a range of populations whose genealogical structures are not captured well by Kingman's $n$-coalescent. $\Lambda$-$n$-coalescents can be seen as the limit process of the discrete genealogies of Cannings models with fixed population size, when time is rescaled and population size $N\to\infty$. As established for Kingman's $n$-coalescent, moderate population size fluctuations in the discrete population model should be reflected by a time-change of the limit coalescent. For $\Lambda$-$n$-coalescents, this has been explicitly shown for only a limited subclass of $\Lambda$-$n$-coalescents and  exponentially growing populations. This article gives a general construction of time-changed $\Lambda$-$n$-coalescents as limits of specific Cannings models with rather arbitrary time changes.
\end{abstract}

\maketitle
\input{main3.tex}

\bibliographystyle{abbrvnat}
\bibliography{popsized}
\end{document}

%% file: main3.tex
\section{Introduction}\label{sec:intro}
The genealogies of samples from populations with highly variant offspring numbers, for instance due to sweepstake reproduction or rapid selection, are not well modelled by Kingman's $n$-coalescent. As a more realistic alternative, multiple-merger coalescents, especially $\Lambda$-coalescents have been proposed, as reviewed in \cite{Tellier2014}, \cite{irwin2016importance} and \cite{eldon2016current}. $\Lambda$-$n$-coalescents, introduced by \cite{Pitman1999,Sagitov1999,donnelly1999particle}, are Markovian processes $(\Pi_t)_{t\geq 0}$, which describe the genealogy of a set of individuals $\{1,\ldots,n\}$. This is done by representing the ancestral lineages present at time $t$ of these individuals by the sets of offspring of each ancestral lineage in the sample. Thus, $(\Pi_t)_{t\geq 0}$ can be defined as a random process with states in the set of partitions of $\{1,\ldots,n\}$ and transitions via merging of blocks (i.e. merging of ancestral lineages to a common ancestor). For a $\Lambda$-$n$-coalescent, the infinitesimal rates of any merger of $k$ of $b$ present lineages is given by $\lambda_{b,k}:=\int_0^1 x^{k-2}(1-x)^{b-k}\Lambda(dx)$, where $\Lambda$ is a finite measure on $[0,1]$. This includes Kingman's $n$-coalescent if $\Lambda$ is the Dirac measure in 0.\\ As in the case of Kingman's $n$-coalescent being the limit genealogy from samples taken from a discrete Wright-Fisher or Moran model, $\Lambda$-$n$-coalescents can be constructed as the (weak) limit of genealogies from samples of size $n$ taken from Cannings models. The limit is reached as population size $N$ goes to infinity and time is rescaled, see \cite{Moehle2001}. Time is rescaled by using $[c^{-1}_N]$ generations in the discrete model as one unit of evolutionary (coalescent) time in the limit, where $c_N$ is the probability that two individuals picked in a generation have the same parent one generation before. In the discrete models, the population size $N$ is fixed across all generations.\\
Only populations in an equilibrium state are described well by models with fixed population sizes. This idealized condition often does not apply to natural populations. In particular, due to fluctuating environmental conditions population sizes are expected to fluctuate likewise. Two standard models of population size changes are timespans of exponential growth or decline, as well as population bottlenecks, where population size drops to a fixed size smaller than $N$ for a timespan on the evolutionary (coalescent) timescale. Such changes are featured in coalescent simulators as \texttt{ms} \citep{Hudson2002} or \texttt{msprime} \citep{msprime}. The latter changes are also the model of population size changes in \texttt{PSMC} \citep{li2011inference} or similar approaches as \texttt{SMC++} \citep{terhorst2017robust}.  For the Wright-Fisher model, which converges to Kingman's $n$-coalescent if population size $N$ is fixed for all generations, the same scaling $c_N^{-1}$ from discrete genealogy to limit is valid for population size changes which maintain a population size of order $N$ at all times, see \cite{griffiths1994sampling} or \cite{Kaj2003}. The resulting limit process is Kingman's $n$-coalescent, whose timescale is (non-linearly) transformed. However, size changes too extreme can yield a non-bifurcating (multiple merger) genealogy, see \citet[Sect. 6.1]{Birkner2009}.\\ 
For $\Lambda$-$n$-coalescents, the link between fluctuating population sizes in the discrete models and the time-change in the coalescent limit is somewhat less established. While conditions for convergence of the discrete genealogies to a limit process are given in \cite{mohle2002coalescent}, no explicit construction of haploid Cannings models leading to an analogous limit, a $\Lambda$-$n$-coalescent with changed time scale, is given. For a specific case, the Dirac $n$-coalescent for an exponentially growing population, such a construction has been given in \cite{Matuszewski2017}, based on the fixed-$N$ Cannings model (modified Moran model) from \cite{Eldon2006}. However, also other $\Lambda$-$n$-coalescents (or Cannings models which should converge to these) with changed time scale have been recently discussed and applied as models of genealogies, see \cite{Spence1549}, \cite{Kato171060}, \cite{alter2016population} and \cite{hoscheit2018multifurcating}. This leads to the goal of this article, which is to extend the approach in \cite{Matuszewski2017} to explicitly give a construction of time-changed $\Lambda$-coalescents as limits of Cannings models with fluctuating population sizes. 
The Cannings models used are modified Moran models, see e.g. \cite{huillet2013extended}, and the Cannings models introduced in \cite{Schweinsberg2003}. The main tool to establish the convergence to the time-changed $\Lambda$-$n$-coalescent is, as in \cite{Matuszewski2017}, applying \citet[Thm. 2.2]{mohle2002coalescent}.\\
For diploid Cannings models, the umbrella model  from \cite{koskela2018robust} gives a general framework to add population size changes, selection, recombination and population structure to the fixed-$N$-model. There, if one only considers population size changes, the limit is a time-changed $\Xi$-$n$-coalescent, a coalescent process with simultaneous multiple mergers. The focus in the present paper is slightly different though, the aim is to explicitly construct Cannings models that converge, after linear time scaling, to a time-changed $\Lambda$-$n$-coalescent, while \cite{koskela2018robust} concentrates on the convergence itself.           

\section{Models and main results}
Cannings models \citep{cannings1974latent,Cannings1975} describe the probabilistic structure of the pedigree (offspring-parent relations) of a finite population in generations $v\in\mathbb{Z}=\{\ldots,2,-1,0,1,2\ldots\}$ with integer-valued population sizes $(N_v)_{v\in\mathbb{Z}}$. The $N_v$ individuals in generation $v$ produce $(\nu^{(v)}_1,\ldots,\nu^{(v)}_{N_v})$ offspring, where $\sum_{i=1}^{N_v}\nu^{(v)}_i=N_{v+1}$ and offspring sizes are exchangeable, i.e. $(\nu^{(v)}_1,\ldots,\nu^{(v)}_{N_v})\stackrel{d}{=}(\nu^{(v)}_{\sigma(1)},\ldots,\nu^{(v)}_{\sigma(N_v)})$ for any permutation $\sigma\in S_{N_v}$. The offspring generation $v+1$ then consists of these individuals in arbitrary order (independent of the parents). The case $N_v=N$ for all $v$ is denoted as the fixed-$N$ case.\\From now on, look at the genealogy of the population in generation 0. For convenience, denote the generations in reverse order by $r=-v$, i.e. if one looks $i$ generations back, this is denoted by $r=i$. The population sizes $N_r$ are defined relative to a reference size $N$, in a way that if $N\to\infty$, also $N_r\to\infty$. From now on, use $N=N_0$. The goal is to establish a limit process of the discrete genealogies for $N\to\infty$. The discrete genealogy of a sample of size $n$ in generation 0 is a random process $(\mathcal{R}^{(N)}_r)_{r\in\mathbb{N}_0}$ with values in the partitions of $\{1,\ldots,n\}$, where $i,j$ are in the same block of $\mathcal{R}^{(N)}_r$ iff they share the same ancestor in generation $r$.\\ 
The terminology from \cite{mohle2002coalescent} is used with slight adaptations. Let $c_{N,r}$ be the probability that two arbitrary individuals in generation $r-1$ have the same ancestor in generation $r$ in the model with reference population size $N$. To clarify, $c_{N,r}$ is the coalescence probability for individuals in generation $r-1$  if population sizes are variable, while $c_N$ denotes the coalescence probability in the fixed$-N$ case. 
Define $F_{N}(s)=\sum_{r=1}^{\left[s\right]}c_{N,r}$ and let 
\begin{equation}\label{eq:def_mathcalG}
\mathcal{G}^{-1}_{N}(t)=\inf\left\lbrace s>0: F_N(s)>t\right\rbrace -1
\end{equation}
be its shifted pseudo-inverse. For $l$ and $a_1,\ldots,a_l\geq 1$, set 
\begin{equation}\label{eq:cann_transprob}
\Phi^{(N)}_l(r;a_1,\ldots,a_l)=\frac{(N_r)_lE\left(\prod_{i=1}^l (\nu_i^{(r)})_{a_i}\right)}{(N_{r-1})_{\sum^l_i a_i}}
\end{equation} 
as the probability that in generation $r-1$, from $\sum_{i=1}^l a_i$ individuals sampled from the Cannings model, specific sets of $a_1\geq\cdots\geq a_l$ individuals each find a common ancestor one generation before (generation $r$), where ancestors of different sets are different. For $l=1,a_1=2$, $c_{N,r}=\Phi^{(N)}_1(r;2)$ See \cite{moehle_robust1998} for details.\\  
Consider a sequence of fixed-$N$ Cannings models for each $N\to\infty$ with $c_N\to 0$ for $N\to\infty$ and transition probabilities $\Phi^{(N)}_l(a_1,\ldots,a_l)$ for a merger of $a_1,\ldots,a_l\geq 2$ individuals , converging to a $\Lambda$-$n$-coalescent $(\Pi_t)_{t\geq 0}$ with infinitesimal transition rates $\phi_l(a_1,\ldots,a_l):=\lambda_{a_1,a_1}1_{\{l=1\}}$ when scaled by $c_N^{-1}$, i.e.
\begin{equation}\label{eq:fixedNcnvmmc}
(\mathcal{R}^{(N)}_{[c_N^{-1}t]}\stackrel{d}{\to} (\Pi_t)_{t\geq 0}
\end{equation} 
in the Skorohod sense for $N\to\infty$. Eq. \ref{eq:fixedNcnvmmc} is satisfied if $$c_N\to 0 \mbox{ and } c_N^{-1}\Phi^{(N)}_l(a_1,\ldots,a_l)\to \phi_l(a_1,\ldots,a_l)$$ for $N\to\infty$, see \cite[Thm. 2.1]{Moehle2001}.
We will establish a  variant of \citet[Corollary 2.4]{mohle2002coalescent} to show convergence of a variety of Cannings models with variable population sizes to (time-changed) $\Lambda$-$n$-coalescents. For this, we need some assumptions. Most importantly, an asymptotically infinite sum needs to be controlled. For this, we introduce a stronger concept of $o$-terms: For a null sequence $(x_N)_{N\in\N}$, a sequence of terms is $o_{\sum}(x_N)$ if summing $O(x^{-1}_N)$ of them still vanishes, e.g. if these summands divided by $x_N$ have the same null sequence majorant.\\
Fix $t>0$. We assume for all $0\leq r \leq [c_{N}^{-1}t]$:
\begin{flushleft}
\begin{itemize}
\item Population size changes of order $N$ leading to a well-defined population size profile in coalescent time, i.e.
\begin{equation}\label{eq:cond_popsize}
\begin{split}
0<N^{-}(t)&:=c_1(t)N  \leq N_r \leq c_2(t)N=:N^{+}(t)<\infty\\ N^{-1} & N_{\lfloor tc_N^{-1}\rfloor}\to\nu(t) \mbox{ for } N\to\infty
\end{split}
\end{equation} 
for positive and finite functions $c_1,c_2,\nu:\mathbb{R}_{\geq 0}\to\mathbb{R}_{>0}$.
 \item
\begin{equation}\label{eq:cond_phi}
\Phi^{(N)}_l(r;a_1,\ldots,a_l)=\Phi^{(N_r)}_l(a_1,\ldots,a_l)+o_{\sum}(c_{N_r})
\end{equation}

\end{itemize}
\end{flushleft}

The first class of Cannings models used to construct time-changed $\Lambda$-$n$-coalescents are modified Moran models. In a modified Moran model, only a single individual has more than one offspring (and may have many offspring). Following \cite{huillet2013extended}, define the modified (haploid) Moran model with fixed population size $N$. Let $(U_N(z))_{z\in\mathbb{N}}$ be i.i.d. random variables with values in $\{2,\ldots,N\}$, let $U_N$ be a r.v. with their common distribution. In each generation $z\in\mathbb{Z}$,
\begin{itemize}
\item One randomly chosen individual has $U_N(z)$ offspring,
\item $U_N(z)-1$ randomly chosen individuals have no offspring,
\item The other $N-U_N(z)$ have one offspring each,
\end{itemize} 
Specific modified Moran models leading to Dirac $n$-coalescents as genealogy limits have been introduced as population models with skewed offspring distributions, see \cite{Eldon2006} and \cite{Matuszewski2017}, for fixed and variable population sizes.\\  For any $\Lambda$-$n$-coalescent with $\Lambda([0,1])=1$ (denoted by $\Lambda\in\mathcal{M}[0,1]$), \citet[Prop. 3.4]{huillet2013extended} shows that there always exist fixed-$N$ modified Moran models such that their rescaled genealogies converge to the $\Lambda$-$n$-coalescent. These can be constructed via a random variable $U'_N$, that is distributed like the merger size of the first merger in a $\Lambda$-$N$-coalescent. As shown in \citep[Eq. 9]{huillet2013extended}, this means
\begin{equation}\label{eq:LambdamodMoran}
P(U'_N=j)=\lambda_N^{-1} \binom{N}{j}E(X^{j-2}(1-X)^{N-j}),\ j\geq 2,
\end{equation}
where $\lambda_N$ is the total transition rate of the $\Lambda$-$N$-coalescent and $X$ has distribution $\Lambda$. \\ 
To add population size changes from generation to generation to the modified Moran model, the relationship between offspring and parent generation needs to be defined. This will be done by adjusting the fixed-$N$ model: If in generation $r$, there are $N_r$ individuals, first run a fixed-$N_r$ modified Moran model, producing $N_r$ (potential) offspring. Let $U_{n,r}$ denote the number of offspring in generation $r-1$ of the multiplying parent in generation $r$.\\ 
If population size declines from generation $r$ to $r-1$, sample $N_{r-1}$ individuals randomly (without replacement) from the $N_{r}$ potential offspring consisting of $U_{N_r}$ offspring of the multiplying parent and $N_{r}-U_{N_r}$ single offspring. If individuals are added to the population (population growth), i.e. \label{p:pedigree}
\begin{equation}\label{eq:def_popsizechange}
d_{N,r}=N_{r-1}-N_r>0,
\end{equation} many individuals, 
to still end up with a modified Moran model one has two options. Additional individuals can be added as further offspring of the already multiplying parent. 
A second option is to add individuals as offspring of the non-reproducing individuals from generation $r$ in the fixed-$N_r$ model. Each originally non-reproducing parent can have one offspring, so this allows to add $U_{N_r}-1\geq 1$ individuals. The number of additional individuals can be divided between these two options, let $A_{n,r}$ denote the individuals added as offspring to the multiplying parent (which means that $d_{N,r}-A_{N,r}$ individuals are added as offspring of non-reproducing parents from the fixed-$N_r$ model). Expressed differently, there are $N_{r-1}=N_r+d_{N,r}$ offspring from which $U_{n,r}=U_{N_r}+A_{N,r}$ share the same parent, while all other offspring are single offspring of other parents (who all differ). See Figure \ref{fig:modMoran} for an example.
\begin{figure}[ht]
\includegraphics[scale=.5]{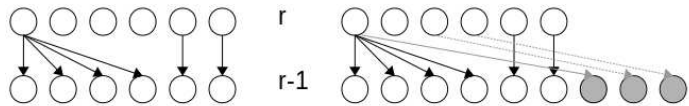}
\caption{Example of allocation of individuals when population size is increasing. Left: Start with a fixed size Moran model with $U_{6}=4$. Right: Population increases by $d_{6,r}=3$, from which $A_{n,r}=1$ individual is allocated to the multiplying parent from generation $r$ in the fixed size model (and 2 to non-reproducing individuals from the fixed size model in generation $r$). This results in $U_{6,r}=5$}
\label{fig:modMoran}
\end{figure}

 While some care has to be taken to not change coalescence probabilities (see Remark \ref{rem:se} for an example), there will be different possibilities to choose $A_{N,r}$. For Dirac-$n$-coalescents with exponential growth (on the coalescent time scale), \cite{Matuszewski2017} used $A_{n,r}=d_{N,r}$. A reasonable approach may also be to set $A_{n,r}$ (close to) proportional to the fraction $U_{N_r}/N_r$ of offspring coming from the multiplying parent: Each of the $d_{N,r}$ added individuals are added to the multiplying parent with probability $U_{N_r}/N_r$ (with the obvious constraint that after $U_{N_r}-1$ individuals are added as offspring of non-reproducing parents, all further individuals need to be added to the multiplying parent). As for the fixed-size models, we consider the genealogy of a sample of $n$ individuals, which is denoted by $(\tilde{\mathcal{R}}^{(N)}_r)_{r\in\N_0}$\\
The main results of the present paper show that the two allocation schemes allow to construct $\Lambda$-$n$-coalescent limits of the genealogies of these modified Moran models if population sizes vary in the discrete models in ways described by Eq. \eqref{eq:cond_popsize}.   
\begin{theorem}\label{thm:modMoran2tcMMCthinned}
Let $\Lambda\in\mathcal{M}[0,1]$ so that $U'_N$ defined by Eq. \ref{eq:LambdamodMoran} satisfies $$E((U'_N)_2)(N-1)^{-1}\nrightarrow 0 \mbox{ for } N\to\infty.$$ Define a modified Moran model for fixed $N$ by $$U_N:=U'_N1_{A_N}+2(1-1_{A_N})$$ for sets $A_N$ s.t. $U'_N$,$1_{A_N}$ are independent and $E((U'_N)_2)P(A_N)((N)_2)^{-1}=N^{-\gamma}$ for $1<\gamma<2$. Let $\nu:\mathbb{R}_{\geq 0}\to \mathbb{R}_{>0}$ be a positive real function. Then, there exist population sizes satisfying Eq. \eqref{eq:cond_popsize} for $\nu$ so that the genealogies $(\tilde{\mathcal{R}}^{(N)}_r)_{r\in\N_0}$ of the modified Moran model with variable population sizes converge $$(\tilde{\mathcal{R}}^{(N)}_{[c_N^{-1} t]})_{t\geq 0}\stackrel{d}{\to}(\Pi_{\mathcal{G}(t)})_{t\geq 0}$$ in the Skorohod-sense, where $\mathcal{G}(t)=\int^t_0 (\nu(s))^{-\gamma} ds$ and $(\Pi_t)_{t\geq 0}$ is a $\Lambda$-$n$-coalescent. In the discrete model, additional individuals can be added in any way so that the resulting model is still a modified Moran model.            
\end{theorem}
For $\Lambda$ not covered by Theorem \ref{thm:modMoran2tcMMCthinned}, one can choose slightly different modified Moran models that converge to a $\Lambda$-$n$-coalescent limit for an arbitrary population size profile on the coalescent time scale. 
\begin{theorem}\label{thm:modMoran2tcMMC}
Fix $\Lambda\in\mathcal{M}[0,1]$ so that $U'_N$ defined by Eq. \eqref{eq:LambdamodMoran} satisfies $$E((U'_N)_2)(N-1)^{-1}\to 0.$$ For fixed population size $N$, define modified Moran models via $U_N=U'_N$. Let $\nu:\mathbb{R}_{\geq 0}\to \mathbb{R}_{>0}$ be a positive function describing the population size profile. Then, there exist population sizes satisfying Eq. \eqref{eq:cond_popsize} for $\nu$ so that the genealogies $(\tilde{\mathcal{R}}^{(N)}_r)_{r\in\N_0}$ of the modified Moran model with variable population sizes fulfill $(\tilde{\mathcal{R}}^{(N)}_{[\mathcal{G}_N^{-1}(t)]})_{t\geq 0}\stackrel{d}{\to}(\Pi_t)_{t\geq 0}$ in the Skorohod-sense, where $(\Pi_t)_{t\geq 0}$ is a $\Lambda$-$n$-coalescent. In the discrete model, additional individuals are added solely as offspring of non-reproducing parents from the fixed-$N_r$ model, unless $E((U_N)_2)\to\infty$ for $N\to\infty$. In that case, they can be added any way that preserves that the model is still a modified Moran model. 
\end{theorem}

\begin{remark}
The condition of $\Lambda([0,1])=1$ in both theorems is not very important: If one scales by $c_2c_N$ instead of $c_N$ for any $c_2>0$, the rescaled discrete genealogies converge to the $c_2\Lambda$-$n$-coalescent. 
\end{remark}

If $c_{N,r}$ and $c_N$ are known well enough, the time-change in Theorem \ref{thm:modMoran2tcMMC} can also be shifted to the limit.
\begin{corollary}\label{cor:beta_thm2}
Let $\Lambda\in\mathcal{M}[0,1]$ be a Beta(a,b)-distribution with $a\in(0,1)$ and $b>0$. Let $\nu:\mathbb{R}_{\geq 0}\to \mathbb{R}_{>0}$.  Then, there exist population sizes satisfying Eq. \eqref{eq:cond_popsize} for $\nu$ so that the genealogies $(\tilde{\mathcal{R}}^{(N)}_r)_{r\in\N_0}$ of the modified Moran model with variable population sizes fulfill $(\tilde{\mathcal{R}}^{(N)}_{[ c_N^{-1}t]})_{t\geq 0}\stackrel{d}{\to}(\Pi_{\mathcal{G}(t)})_{t\geq 0}$ in the Skorohod-sense, where $\mathcal{G}(t)=\int^t_0 (\nu(s))^{a-2} ds$ and $(\Pi_t)_{t\geq 0}$ is a $Beta(a,b)$-$n$-coalescent. In the discrete model, additional individuals can be added in any way so that the resulting model is still a modified Moran model.            
\end{corollary} 

The specific models used in each of the two theorems are not the only possibilities of modified Moran models with variable population sizes to converge to $\Lambda$-$n$-coalescents. For instance, if one only allows certain population size changes, one can also use the modified Moran model used in Theorem \ref{thm:modMoran2tcMMC} for some $\Lambda$ covered by Theorem \ref{thm:modMoran2tcMMCthinned}.
\begin{corollary}\label{cor:betanonthin}
Let $\Lambda\in\mathcal{M}[0,1]$ be a Beta(a,b)-distribution with $a\in(1,2)$ and $b>0$. Consider an exponentially growing modified Moran model population on the coalescent time scale, i.e. $\nu(t)=\exp(-\rho t)$ Then, there exist population sizes satisfying Eq. \eqref{eq:cond_popsize} for $\nu$ so that the genealogies $(\tilde{\mathcal{R}}^{(N)}_r)_{r\in\N_0}$ of the modified Moran model with variable population sizes fulfills $(\tilde{\mathcal{R}}^{(N)}_{[c_N^{-1}t]})_{t\geq 0}\stackrel{d}{\to}(\Pi_{\mathcal{G}(t)})_{t\geq 0}$ in the Skorohod-sense, where $\mathcal{G}(t)=\int^t_0 (\nu(s))^{a-2} ds$ and $(\Pi_t)_{t\geq 0}$ is a $\Lambda$-$n$-coalescent. In the discrete model, additional individuals can be added in any way so that the resulting model is still a modified Moran model.                         
\end{corollary} 
Finally, for the classic Moran model, we can establish
\begin{proposition}\label{prop:Moranlimit}
For the standard Moran model and a population size profile $\nu:\mathbb{R}_{\geq 0}\to \mathbb{R}_{>0}$, there exist population size changes allowed by Eq. \eqref{eq:cond_popsize} so that $(\tilde{\mathcal{R}}^{(N)}_{[c_N^{-1}t]})_{t\geq 0}\stackrel{d}{\to}(\Pi_{\mathcal{G}(t)})_{t\geq 0}$ in the Skorohod-sense, where $\mathcal{G}(t)=\int^t_0 (\nu(s))^{-2} ds$. Individuals are added only as offspring of non-reproducing parents (in the fixed-$N_r$ model) if the population size increases.
\end{proposition} 
For Beta$(2-\alpha,\alpha)$-$n$-coalescents, $1\leq\alpha<2$, genealogies sampled from the fixed-$N$ Cannings models introduced in \cite{Schweinsberg2003} also converge weakly to these Beta coalescent processes (after rescaling of time) for $N\to\infty$. \\
This model (for fixed population size $N$) lets each individual in any generation $r$ produce a number of (potential) offspring $X^{(r)}_i$, i.i.d. across individuals and generations, distributed as a tail-heavy random variable $X$ with $E(X)=\mu>1$, i.e. 
\begin{equation}\label{eq:def_schwvar}
P(X\geq k)\sim Ck^{-\alpha} \mbox{ on } \N  \mbox{($C>0$ constant, $1<\alpha<2$)}.
\end{equation}
Then, $N$ offspring are chosen to form the next generation. If less than $N$ offspring are produced, the missing next generation individuals are arbitrarily associated with parents. Here, this is done by randomly choosing a parent, which preserves exchangeability and makes the model a Cannings model. The genealogies of a sample of size $n$ converge for $N\to\infty$ and time rescaled by $c_N^{-1}$ to the Beta($2-\alpha$,$\alpha$)-$n$-coalescent, see \citet[Thm. 4]{Schweinsberg2003}. \\
This model can very easily extended to variable population sizes by just sampling from the potential offspring. The tail-heavy distributions used produce, asymptotically for $N\to\infty$, enough potential offspring to cover growing population sizes of order $N$ as allowed by Eq. \eqref{eq:cond_popsize}.  
\begin{lemma}\label{lem:popsizechange_schwein}
Let $d_{N,r}:=N_{r-1}-N_r>0$. Assume that for any fixed $t$, for all $r\leq c^{-1}_Nt$ there exists a null sequence $(d_N)_{N\in\N}$ with $d_{N,r}/N\leq d_N$ for $N\to\infty$. Then, $P(\sum_{i=1}^{N_r} X^{(r)}_{i}<N_{r-1})\leq A^{N}$ with $N=N_0$ and $A<1$.
\end{lemma}
This gives us an alternative Cannings model with variable population sizes to define time-changed Beta coalescents as the limit of its discrete genealogies. 
\begin{theorem}\label{prop:schweinspopsized}
Consider the Cannings model coming from sampling from potential i.i.d. offspring following Eq. \eqref{eq:def_schwvar} with parameter $\alpha\in[1,2)$. For any $\nu:\mathbb{R}_{\geq 0}\to \mathbb{R}_{>0}$, there exist variable population sizes $(N_r)_{r\in\N_0}$ fulfilling \eqref{eq:cond_popsize} for $\nu$ so that the discrete $n$-coalescents converge 
 $(\tilde{\mathcal{R}}^{(N)}_{[c_N^{-1}t]})_{t\geq 0}\stackrel{d}{\to}(\Pi_{\mathcal{G}(t)})_{t\geq 0}$ in the Skorohod-sense, where $\mathcal{G}(t)=\int^t_0 (\nu(s))^{1-\alpha} ds$ and where $(\Pi_t)_{t\geq 0}$ is the Beta($2-\alpha,\alpha$)-$n$-coalescent.
\end{theorem} 

The time-change function $\mathcal{G}(t)$, which appears in Theorem \ref{thm:modMoran2tcMMCthinned}, Corollary \ref{cor:betanonthin}, Propositions \ref{prop:Moranlimit} and \ref{prop:schweinspopsized}	 simplifies considerably for exponential growth on the coalescent time scale, i.e. $\nu(t)=exp(-\rho t)$ for $\rho>0$ in Eq. \eqref{eq:cond_popsize} (corresponding to population sizes given by $N_{r-1}=\lfloor N_r(1-c_N\rho)\rfloor$ for $r\in\N$).
\begin{corollary}\label{cor:gompertz}
For a population size profiles of exponential growth (on the coalescent scale) with growth rate $\rho$ and for $c_N=cN^{-\gamma}+o(N^{-\gamma})$ for $\gamma>0$, the time-change function $\mathcal{G}$ has the form
\begin{equation}\label{eq:mathcalGfexp}
\mathcal{G}(t)=\int_0^t e^{\rho\gamma s}ds=(\rho\gamma)^{-1}(e^{\rho\gamma t}-1).
\end{equation}
This implies that the waiting time between coalescent events are Gompertz distributed  with parameters $a=\lambda_be^{\rho\gamma t_0}$ and $b=\rho\gamma$., i.e. the waiting time $T$ for the next coalescence event, given the last coalescence at $t_0$ into $b$ lineages, fulfills 
$$P_{t_0}(T\leq t)=1-e^{\lambda_b(\rho\gamma)^{-1}(e^{\rho\gamma (t+t_0)}-e^{\rho\gamma t_0})}=1-e^{\lambda_b(\rho\gamma)^{-1}e^{\rho\gamma t_0}/e^{\rho\gamma t}-1}.$$ 
\end{corollary}
\begin{remark}
It is well-known that for Kingman's $n$-coalescent with exponential growth, waiting times for coalescence events follow a Gompertz distribution, e.g. see \citet[Eq. 5]{slatkin1991pairwise}, \cite{polanski2003note}. For time-changed Dirac coalescents appearing as limits of modified Moran models with $\nu(t)=exp(-\rho t)$, Eq.  \eqref{eq:mathcalGfexp} appeared in \cite{Matuszewski2017}. 
\end{remark}
\section{Discussion}\label{sec:disc}
As for the Wright-Fisher model, genealogies of samples taken from  (haploid) modified Moran and other Cannings models can be approximated by a time-change of their limit coalescent process, when the population sizes of the discrete models are fluctuating, but are always of the same order of size. As for models with fixed population size, time intervals of $[c^{-1}_N t]$ generations in the discrete model correspond to a time interval of length 1 in the continuous time limit. The approach of this study was to build on existing Cannings models that converge for fixed population size to the $\Lambda$-$n$-coalescent and just change the population sizes gradually from generation to generation, which includes adjusting parent-offspring allocation between generations. This raises the question whether the used Cannings models and the adjustment of ancestral relationships have biological interpretations and are a reasonable model for at least some real populations.
\subsection{Interpretation of the Cannings models and allocation schemes used}
The modified Moran models used to construct a time-changed $\Lambda$-$n$-coalescent with $\Lambda([0,1])=1$ (defined via Eq. \ref{eq:LambdamodMoran}, introduced in \cite{huillet2013extended}) can be described as follows (for fixed $N$): On top of a standard Moran model choice of one parent $M$ with two offspring and one individual in the parent generation with no offpring, there is a random probability $X$ for each other individual in the parent generation to not have offspring in the next generation. $X$ is drawn from $\Lambda$, potentially only activated in a given generation with a low probability $N^{-\gamma}$, $\gamma\in(1,2)$. From the individuals that have offspring, all but $M$ reproduce once, and $M$ replaces itself and all non-reproducing individuals by its offspring. These models capture the concept of sweepstake reproduction \citep{Hedgecock2011}, though the assumption of a single individual with more than one offspring is rather artificial. For a non-random $X$ and large families appearing occasionally at rate of order $N^{-\gamma}$, this model is very similar to the discrete modified Moran model from \cite{Eldon2006} used to describe sweepstake reproduction (and that was used in \cite{Matuszewski2017} as a basis to construct a time-changed Dirac $n$-coalescent). Both models lead to the same Dirac coalescent limit and have the same time rescaling order $c_N^{-1}$. In \cite{Eldon2006}, instead of randomly choosing individuals to not reproduce with probability $X$, a fixed number of $\approx NX-2$ individuals are chosen at random to not reproduce on top of the Moran choice (again with a small probability in each generation for this to happen) For random $X$, similar models also appear in \cite{hartmann2018large} and \cite{eldon2012age}.\\
The other class of Cannings models used to capture skewed offspring distributions, defined via Eq. \eqref{eq:def_schwvar}, lead to the specific class of Beta($2-\alpha$,$\alpha$)-$n$-coalescents. They have been proposed as a model of type-III survivorship, where all individuals produce many offspring with a high juvenile mortality, see e.g. \citet[Sect. 2.3]{Steinruecken2013}, also leading to sweepstake-like phenomena. While both classes of Cannings models allow the Bolthausen-Sznitman $n$-coalescent ($\Lambda=$Beta$(1,1)$) as a possible limit model, the discrete models used to explicitly construct it are not based on modelling a directed selection process due to selective advantages of certain ancestral lineages. Thus, the results do not answer whether adding population size changes to a model of rapid selection or genetic draft as in \cite{Desai2013,Neher2013a,Schweinsberg2017} also leads to its rescaled genealogies being described by a time-changed Bolthausen-Sznitman $n$-coalescent.\\
To construct time-changed $\Lambda$-$n$-coalescent as limits of genealogies in modified Moran models, the approach here is to  adjust fixed-$N$ modified Moran models for growing or decreasing population sizes. Sampling the next generation from the fixed-$N$ offspring when there is population decline maintains on average the ratio between the large family $U_N$ and the rest off the individuals. This means that the population decrease, e.g. due to less resources available, has the same chance to affect each offspring of the fixed-size model. Additional individuals can be added to the family of the multiplying parent or by allowing parents with no offspring from the fixed-$N$ allocation scheme to have exactly one offspring. For some sequences of modified Moran models, any partition of additional individuals to these two allocation forms is possible, e.g. allocate them randomly to the multiplying parent (with $U_{N_r}$ offspring) from the fixed-size model with probability $U_{N_r}/N_r$ (with the constraint that we cannot add more than $U_{N_r}-1$ individuals to non-reproducing parents). The merit of this random allocation is that it is trying to maintain the ratio $U_{N_r}/N_r$ from the fixed-size model. As for sampling a smaller number of individuals, this describes that population size increase, e.g. due to more resources available, follows (approximately and on average) the sweepstake pattern of the fixed-$N$ model. From a biological viewpoint, other allocation schemes can also be interpreted: Adding the additional offspring completely to the largest family, as done in \cite{Matuszewski2017}, could describe a scenario where new resources become available and only the multiple-offspring parent can claim them for its offspring. In contrast, adding individuals as single offspring of non-reproducing parents from the fixed-size model  relaxes the (viability) ``selection'' pressure of the modified Moran model by allowing more non-multiplying parents (resp. their offspring) to survive, e.g. due to the additional resources. For the models covered in Theorem \ref{prop:schweinspopsized} from \cite{Schweinsberg2003}, population size changes in either direction are modelled by sampling from a pool of more individuals than the current population size, thus additional or decreasing resources affect the offspring of different parents in the same way.
%

\subsection{Influence of the choice of Cannings model on the limit}
Many results in the present paper allow to scale the time in the discrete models with $c_N^{-1}$ as in the fixed $N$ case so that the scaled genealogies converge to a time-changed $\Lambda$-$n$-coalescent $(\Pi_{\mathcal{G}(t)})_{t\geq 0}$. This time-change $\mathcal{G}(t)$ depends both on the population size profile $\nu$ on the coalescent time scale from Eq. \ref{eq:cond_popsize} and the (asymptotic properties of) the coalescence probabilities $c_N$, i.e. how many discrete generation correspond to one unit of coalescent time. For instance, consider an exponentially growing population (on the coalescent time scale, $\nu(t)=\exp(-\rho t)$ for $\rho > 0$) and two different models leading to a time-changed $Beta(2-\alpha,\alpha)$-$n$-coalescent ($\alpha\in(1,2)$): the ones from Corollary \ref{cor:beta_thm2} and Proposition \ref{prop:schweinspopsized}. From Eq. \ref{eq:mathcalGfexp}, we see that $\mathcal{G}$ depends on the product $\gamma\rho$. For the model from Corollary \ref{cor:betanonthin}, $\gamma=\alpha$ and for the one from Proposition \ref{prop:schweinspopsized}, it is $\gamma=\alpha-1$. Thus, the exact same time-changed $\Lambda$-$n$-coalescent can appear as limit model for genealogies with different population size profiles on the coalescent time scale. As already discussed in \citep{Matuszewski2017} in the case of time-changed Dirac-$n$-coalescents, this poses a problem for inference: If one wants to infer $\rho$ directly (instead of the compound parameter $\gamma\rho$), $\gamma$ has to be known. This means that specifying/identifying the Cannings model leading to the limit process would be necessary to directly estimate $\rho$. This is very similar to the effect that e.g. Watterson's  estimator only allows to estimate the mutation rate on the coalescent time scale, and not the mutation rate in one generation, see e.g. \citet[p. 2627]{Eldon2006}. Another example for different $\nu$ leading to the same time-scaled coalescent limit for different Cannings models is given by the genealogy limit from the Wright-Fisher model and the (usual) Moran models. It is well known, see e.g. \cite{griffiths1994sampling}, that the rescaled genealogy of a sample from a Wright-Fisher model with population size profile $\nu$ converges to Kingman's $n$-coalescent with time change $\mathcal{G}$ as in Eq. \eqref{eq:compute_mathcalG} with $\gamma=1$. However, for the classic Moran model, Prop. \ref{prop:Moranlimit} shows that Eq. \eqref{eq:compute_mathcalG} holds with $\gamma=2$.\\
For families of Cannings models, if the coalescence probability $c_N$ is of order $\log(N)^{-1}$, a curious phenomenon appears: Population size changes of order $N$ do not even alter the limit genealogy. An example is the model from Proposition \ref{prop:schweinspopsized} for the Bolthausen-Sznitman $n$-coalescent ($\Lambda=Beta(1,1)$). One can interpret this for a population described by the model as follows: Even instantaneous bottlenecks or expansions do not influence the effect that a very large family appearing in a generation has on the genealogy. How the population reproduces, i.e. how the offspring distributions compare between different parents, is thus fully controlling the genealogy, regardless of changes that alter the population sizes overall, e.g. changes in range and/or resources.

\section{Proofs}
This section contains the proof of the presented statements as well as some further remarks.
\subsection{Converging to a time-changed coalescent - sufficient conditions} 

First, recall this special case of \citet[Thm. 2.2]{mohle2002coalescent} 

\begin{corollary}\label{cor:moe2.2easy}
If we satisfy, for any fixed t, 

\begin{align}
&\lim_{N\to\infty} \inf_{1\leq r\leq \mathcal{G}^{-1}_{N}(t)} N_r =\infty,\ \ \lim_{N\to\infty} \sup_{1\leq r\leq \mathcal{G}^{-1}_{N}(t)} c_{N,r}=0,\label{eq:moe02conds}\\
& \lim_{N\to\infty} \sum_{r=1}^{\mathcal{\mathcal{G}}^{-1}_{N}(t)} \Phi^{(N)}_l(r;a_1,\ldots,a_l)=q_{a_1,\ldots,a_l}t<\infty, a_1\geq\ldots\geq a_l\geq 2 \label{eq:moe02conds2}
\end{align}  
the discrete-time coalescent $(\tilde{\mathcal{R}}^{(N)}_{[\mathcal{G}_N^{-1}(t)]})_{t\geq 0}$, so rescaled in time, converges in distribution (Skorohod-sense) to a continuous-time Markov chain with transition function $\exp({Qt})$, where $Q$ is a transition rate matrix with entries $q_{a_1,\ldots,a_l}$, $a_1\geq\ldots\geq a_l\geq 2$ (so diagonal entries are the negative row sums of the other entries). 
\end{corollary}
\begin{remark}
When compared to the original formulation of \citet[Thm 2.2]{mohle2002coalescent}, the limit here can be described as a homogeneous Markov chain with rate matrix $Q$ instead of the more complicated original description of the transition probabilities as a product integral of matrix-valued measures. This directly follows from the stronger condition \eqref{eq:moe02conds2}, where for \citet[Thm 2.2]{mohle2002coalescent} to hold only convergence and not linear dependence on $t$ is needed. Indeed, if \eqref{eq:moe02conds2} holds, the value $\Pi((0,t])$ of the product measure $\Pi$ in \citet[Thm. 2.2]{mohle2002coalescent} has the form $Qt$. This is stated on \citet[p. 209]{mohle2002coalescent}, see also Eq. (24) therein. Then, the form of the transition function is described on \citet[p. 203]{mohle2002coalescent}.
\end{remark}  
Now, recall the conditions \eqref{eq:cond_phi}, \eqref{eq:cond_popsize} and \eqref{eq:cond_speed_cN}. We need some further observations and reformulations. 
\begin{itemize}
\item If we choose $l=1$ in Eq. \ref{eq:cond_phi}, we have for $M_1(t),M_2(t)\in(0,\infty)$
\begin{equation}\label{eq:cond_cnr}
M_1(t)c_{N} + o_{\sum}(c_N) \leq c_{N,r} \leq M_2(t)c_{N} + o_{\sum}(c_N)
\end{equation}
for $N\to\infty$.
\item Controlling the speed of convergence of $c_N$: If Eq. \ref{eq:cond_popsize} is satisfied, there exist $M_1(t),M_2(t)\in(0,\infty)$ with
\begin{equation}\label{eq:cond_speed_cN}
M_1(t)\leq \frac{c_{N_r}}{c_N}\leq M_2(t)  
\end{equation}
\item If Condition \eqref{eq:cond_speed_cN} holds, \eqref{eq:cond_phi} can be equivalently formulated with $c_N$ instead of $c_{N_r}$.
\item If \eqref{eq:cond_popsize} is satisfied, Condition \eqref{eq:cond_speed_cN} is also satisfied if $c_N=f(N)$, where $f$ is regularly varying (at $\infty$).
\end{itemize}

Now, we establish an easy-to-verify variant of \citet[Corollary 2.4]{mohle2002coalescent}.
\begin{lemma}\label{lem:cannpopsizetolambda}
Consider a sequence of Cannings models with reference size $N=N_0$ and variable population size $(N_r)_{r\geq 0}$ which fulfill conditions \eqref{eq:cond_popsize},\eqref{eq:cond_phi},\eqref{eq:cond_speed_cN}, $\lim_{N\to\infty}c_N = 0$ and whose genealogies, of a sample of size $n$, are in the domain of attraction of a $\Lambda$-$n$-coalescent $(\Pi_t)_{t\geq 0}$ (rescaled by $c_N^{-1}$). Then, Corollary \ref{cor:moe2.2easy} can be applied, so $(\tilde{\mathcal{R}}^{(N)}_{[\mathcal{G}_N^{-1}(t)]})_{t\geq 0}\stackrel{d}{\to}(\Pi_t)_{t\geq 0}$ in the Skorohod-sense.\\
If furthermore $\mathcal{G}^{-1}(t):=\lim_{N\to\infty} \mathcal{G}_N^{-1}(t)c_N$ exists, we have, with $\mathcal{G}=(\mathcal{G}^{-1})^{-1}$,
\begin{equation}\label{eq:lineartimescale_limitcoal}
(\tilde{\mathcal{R}}^{(N)}_{[t/c_N]})_{t\geq 0}\stackrel{d}{\to} (\Pi_{\mathcal{G}(t)})_{t\geq 0}
\end{equation}  
for $N\to\infty$ 
\end{lemma} 
\begin{proof}
Size changes of order $N$ satisfy the first part of Condition \eqref{eq:moe02conds}. Its second part is then satisfied by \eqref{eq:cond_cnr}, which in turn is satisfied due to \eqref{eq:cond_phi} and \eqref{eq:cond_speed_cN}. Also due to \eqref{eq:cond_cnr}, $F_N$ is bounded by 

\begin{equation}\label{eq:generalFNbounds}
[s]M_1(t')c_N+[s]o(c_N)\leq F_{N}(s)\leq [s]M_2(t')c_N+[s]o(c_N)
\end{equation} 
for $N\to\infty$ and $[s]\leq c_{N}^{-1}t'$ and thus its pseudo-inverse by

$$\frac{t}{M_2(t')c_N}+\frac{o(c_N)}{c_N}-1 \leq \mathcal{G}_N^{-1}(t)\leq \frac{t}{M_1(t')c_N}+\frac{o(c_N)}{c_N}-1$$ with an appropriate $t'\geq t$. This implies that the time change function $\mathcal{G}_N^{-1}$ for the discrete models in Corollary \ref{cor:moe2.2easy} is of order $c^{-1}_N$. Knowing this, we compute 

\begin{align*}
\sum_{r=1}^{\mathcal{G}^{-1}_{N}(t)} \Phi^{(N)}_l(r;a_1,\ldots,a_l)\stackrel{\eqref{eq:cond_phi}}{=}& \sum_{r=1}^{\mathcal{G}^{-1}_{N}(t)}\underbrace{\Phi^{(N_r)}_l(a_1,\ldots,a_l)c_{N_r}^{-1}}_{\to \phi_l(a_1,\ldots,a_l)}c_{N_r}+\sum_{r=1}^{\mathcal{G}^{-1}_{N}(t)} o_{\sum}(c_N) \\ \stackrel{\eqref{eq:cond_phi}}{=}&\phi_l(a_1,\ldots,a_l)\underbrace{\sum_{r=1}^{\mathcal{G}^{-1}_{N}(t)}c_{N,r}}_{=F_N(\mathcal{G}^{-1}_{N}(t))}+ O(1)\underbrace{\sum_{r=1}^{\mathcal{G}^{-1}_{N}(t)} o_{\sum}(c_N)}_{\to 0} \\\to &\phi_l(a_1,\ldots,a_l)t
\end{align*}
for $N\to\infty$
The second equation is valid due to the uniform convergence of\\ $\Phi^{(N_r)}_l(a_1,\ldots,a_l)c_{N_r}^{-1}$ in $r$ for $N\to\infty $ ($N_r$ is bounded from below on the timescale used). This allows to pull out $\phi_l(a_1,\ldots,a_l)$. This shows that condition \ref{eq:moe02conds2} is satisfied and thus  establishes the convergence of     
$(\tilde{\mathcal{R}}^{(N)}_{[\mathcal{G}_N^{-1}(t)]})_{t\geq 0}$ to the same $\Lambda$-$n$-coalescent as the fixed-size model.
Eq. \ref{eq:lineartimescale_limitcoal} follows as described in \citet[Sec. 4]{moehle_robust1998}. 
\end{proof}

The next step is to establish a special case of Lemma \ref{lem:cannpopsizetolambda} which only considers modified Moran models with changing population sizes.
\begin{remark}\label{rem:se}
Depending on the magnitude of a population size increase, adding individuals as further offspring of the multiplying parent from the fixed-size modified Moran model can strongly increase coalescence probabilities. For instance, for a population expansion of size $Nm$, if one just expands by adding $d_{N,r}=Nm$ to the offspring number of the individual with multiple offspring in a single generation, the coalescence probability for this generation is dominated by the population size change. Then $U_{N,r}\geq Nm$, leading to $c_{N,r}=\frac{E((U_{N,r})_2)}{(N_{r-1})_2}\geq \frac{(Nm-1)^2}{(N_{r-1})_2}=O(1)\nrightarrow 0$ for $N\to\infty$. Thus, from generation $r-1$ to $r$, coalescence is still happening with positive probability for $N\to\infty$, which shows that a potential limit coalescent cannot just be a (non-degenerately) time-changed $\Lambda$-$n$-coalescent, a continuous-time (inhomogeneous) Markovian process. This has an implication for modelling of real populations: The genealogy of a sudden population expansion, happening at a specific generation, where a single genotype/individual is responsibly for the population growth, is not given by a time-changed continuous-time $\Lambda$-$n$-coalescent.
\end{remark}

\begin{remark}\label{rem:lintsc}
The condition for Eq. \eqref{eq:lineartimescale_limitcoal} to hold is a weak condition, since $\mathcal{G}_N^{-1}(t)$ is of order $c^{-1}_N$. Additionally, the linear scaling in  \eqref{eq:lineartimescale_limitcoal} makes it easy to introduce a mutation structure. Let mutation be introduced in the discrete model by allowing mutations from parent to offspring with a rate $\mu_N$. If $\mu_Nc_N^{-1}\to \theta$ for $N\to\infty$, the mutations on the time-scaled $\Lambda$-$n$-coalescent are given by a Poisson point process with homogeneous intensity $\theta$. 
\end{remark}

We recall some properties of fixed-$N$ modified Moran models.
\begin{lemma}\label{lem:modMoran_prop} \
\begin{itemize}
\item[(i)] For $N\to\infty$: $U_N/N\stackrel{d}{\to}0$  is equivalent to $c_N=\frac{E((U_N)_2)}{(N)_2}\to 0$ 
\item[(ii)] If $c_N\to 0$ for $N\to\infty$, the genealogies in the modified Moran models converge, with a rescaling of time by $c_N^{-1}$, to a $\Lambda$-$n$ coalescent if  
\begin{equation}\label{eq:modMoran2Lambda_fixN}
\lim_{N\to\infty}c_N^{-1}\Phi^{(N)}_l(a_1,\ldots,a_l)=\lim_{N\to\infty}1_{\{l=1\}}\frac{E((U)_{a_1})}{(N)_{a_1}c_N}=\int^1_0 x^{a_1-2}\Lambda(dx)1_{\{l=1\}}
\end{equation}
\item[(iii)] If $U'_N$ is distributed as in Eq. \eqref{eq:LambdamodMoran}
 \begin{equation}\label{eq:3M_factmom}
E((U'_N)_k)=\frac{(N)_k}{\lambda_N}E(X^{k-2})
\end{equation}
for all $k\geq 2$. 
\end{itemize} 
\end{lemma}
\begin{proof}
(i) from \citet[Lemma 3.2]{huillet2013extended}, (ii) from \citet[Theorem 3.3]{huillet2013extended}, (iii) from \citet[Eq. 10]{huillet2013extended}
\end{proof}

The following proposition provides criteria for genealogies in modified Moran models with fluctuating population sizes to converge to a $\Lambda$-$n$-coalescent after a suitable time change.  
 
\begin{proposition}\label{lem:3Mconv}
Consider a fixed-$N$ modified Moran model so that $U_N/N\stackrel{d}{\to} 0$ for $N\to\infty$ and that \eqref{eq:modMoran2Lambda_fixN} holds for a finite measure $\Lambda$ on $[0,1]$. From this, construct a modified Moran model with varying population sizes $(N_r)_{r\geq 0}$ which satisfy the following conditions. Assume that Eq. \eqref{eq:cond_popsize}, \eqref{eq:cond_speed_cN} are satisfied.  Assume further $d_{N,r}/N_r\leq d_N\to 0$ for $N\to\infty$. Let $A_{n,r}$ be the number of individuals in generation $r-1$ allocated as offspring of the multiplying parent of the fixed-$N_r$ model from generation $r$. If $P(A_{N,r}>0)>0$, further assume $\frac{E(U_N)}{E((U_N)_2)}\to 0$ and $A_{N,r}\leq c_4E(U_N)$ for a constant $c_4>0$ and $N\to\infty$. Additionally, assume $d_{N,r}-A_{n,r}\leq \min\{i:P(U_{N_r}=i)>0\}-1$.\\
Based on the fixed-size modified Moran model and $(N_r)_{r\in\N}$  define a modified Moran model with population sizes $(N_r)_{r\in\N}$ and offspring variable $U_{N,r}=U_{N_r}+A_{N,r}$ for all $r\in\N$.\\ 
Then, 
 $(\tilde{\mathcal{R}}^{(N)}_{[\mathcal{G}_N^{-1}(t)]})_{t\geq 0}\stackrel{d}{\to}(\Pi_t)_{t\geq 0}$ in the Skorohod-sense, where $(\Pi_t)_{t\geq 0}$ is the $\Lambda$-$n$-coalescent limit for the fixed-$N$ modified Moran model. 
\end{proposition}

\begin{proof}
This is shown by applying Lemma \ref{lem:cannpopsizetolambda}. All conditions but Eq. \eqref{eq:cond_phi} of it are clearly fulfilled under the assumptions of the proposition currently proven, see also Lemma \ref{lem:modMoran_prop}.\\ 
To show \eqref{eq:cond_phi}, first assume $d_{N,r}\geq 0$. Then,
\begin{align}\label{eq:fact_osumcn}
&\frac{E((U_{N_r}+A_{N,r})_{a_1})}{(N_r)_{a_1}c_N}=\sum_{k=1}^{a_1}s(a_1,k)\sum_{l=0}^k\binom{k}{l}\frac{E(U_{N_r}^lA_{N,r}^{k-l})}{(N_r)_{a_1}c_N}\nonumber\\=&\frac{E((U_{N_r})_{a_1})}{(N_r)_{a_1}c_N}+\sum_{k=1}^{a_1}s(a_1,k)\underbrace{\sum_{l=0}^{k-1}\binom{k}{l}\frac{E(U_{N_r}^lA_{N,r}^{k-l})}{(N_r)_{a_1}c_N}}_{\to 0}
\stackrel{\eqref{eq:modMoran2Lambda_fixN}}{\to} \int_0^1 x^{k-2}\Lambda(dx) 
\end{align}
uniformly in $r$ where $s(n,k)$ are Stirling numbers of the first kind (convergence of the first summand is at least as fast as for $N^-(t)$). 
To see that the inner sum in Eq. \eqref{eq:fact_osumcn} vanishes, observe that, for $0\leq r \leq c_{N}^{-1}t$ and $0\leq l< k<a_1$, 
\begin{align*}
0\leq & \frac{E(U_{N_r}^lA_{N,r}^{k-l})}{(N_r)_{a_1}c_N} \leq \frac{E(U^l_{N_r})c_4^{k-l}(E(U_N))^{k-l}N_r^{a_1}}{N_r^{a_1}c_N (N_r)_{a_1}} \leq  \frac{c_4^{k-l}E(U_{N})^{k-l}N_r^{a_1}}{(c_1(t))^{a_1-l}N^{a_1-l} c_N (N_r)_{a_1}}\\ \leq & \frac{c_4^{k-l}E(U_N)a_1^{a_1}}{(c_1(t))^{a_1-l}N^2c_N a_1!}\leq \frac{c_4^{k-l}(N)_2E(U_N)a_1^{a_1}}{(c_1(t))^{a_1-l}N^2E((U_N)_2)a_1!} \to 0 
\end{align*}
for $N\to\infty$.
The second inequality comes from plugging in the bound for $A_{N,r}$, the third from Eq. \eqref{eq:cond_popsize} and omitting factors $U_{N_r}/N_r\leq 1$. For the fourth inequality, we use that $\frac{N_r^{a_1}}{(N_r)_{a_1}}\leq \frac{a_1^{a_1}}{a_1!}$ follows from the fact that $x\mapsto\frac{x^{a_1}}{(x)_{a_1}}$ decreases for $x\geq a_1$ and that $U_N/N\leq 1$.\\
Regardless of the allocation of the new individuals, the population model is a modified Moran models with a single multiplying parent. Thus, to show Eq. \eqref{eq:cond_phi} one only needs to show $\Phi_1^{(N)}(r;a_1)=\Phi_1^{(N_r)}(a_1)+o_{\sum}(c_N)$ for $0\leq r \leq c_{N}^{-1}t$. Compute further
\begin{align*}
&\Phi^{(N)}_l(r;a_1)=E((U_{N_r}+A_{N,r})_{a_1})/(N_r+d_{N,r})_{a_1}\\
=&  \frac{(N_r)_{a_1}}{(N_r+d_{N,r})_{a_1}}E((U_{N_r}+A_{N,r})_{a_1})/(N_r)_{a_1}\\ \stackrel{(*)}{=}&(1-o(1))^{-1}\left(\frac{E((U_{N_r})_{a_1})}{(N_r)_{a_1}}+o_{\sum}(c_N)\right)=\Phi^{(N_r)}_1(a_1)+o_{\sum}(c_N).
\end{align*}
Equation $(*)$ follows from Eq. \eqref{eq:fact_osumcn} and, for the first factor, from $N^{-1}d_{N,r}\leq d_N$ being a null sequence.\\
Now, consider $d_{N,r}<0$. Then, we get the offspring population by sampling $N_{r-1}$ individuals out of $N_r$, from which $U_{N_r}$ share one common parent. Thus, this is again a modified Moran model, where  
$U_{n,r}$ is conditionally hypergeometrically distributed with $P(U_{N,r}=k|U_{N_r})=\frac{\binom{U_{N_r}}{k}\binom{N_r-U_{N_r}}{N_{r-1}-k}}{\binom{N_r}{N_{r-1}}}$. 
Using the factorial moment of the hypergeometric distribution leads to
\begin{align*}
& \Phi^{(N)}_1(r;a_1)= E\left(\frac{(U_{n,r})_{a_1}}{(N_{r-1})_{a_1}}\right)=\frac{E((U_{n,r})_{a_1}\vert U_{N_r})}{(N_{r-1})_{a_1}}=\frac{(N_{r-1})_{a_1}}{(N_{r-1})_{a_1}}\frac{E((U_{N_r})_{a_1})}{(N_r)_{a_1}}
\\=&\Phi^{(N_r)}(a_1),
\end{align*} 
\end{proof}

Finally, the following lemma provides the arguments necessary to shift the time-change $\mathcal{G}_N^{-1}$ from pre-limit to limit.

\begin{lemma}\label{lem:shifttc}
Assume that for a Cannings model with variable population sizes $(N_r)_{r\in\N}$, the discrete $n$-coalescents satisfy $(\tilde{\mathcal{R}}^{(N)}_{[\mathcal{G}_N^{-1}(t)]})_{t\geq 0}\stackrel{d}{\to}(\Pi_t)_{t\geq 0}$ in the Skorohod-sense for $N\to\infty$, where $(\Pi_t)_{t\geq 0}$ is a $\Lambda$-$n$-coalescent and $\mathcal{G}^{-1}_N$ is defined via Eq. \eqref{eq:def_mathcalG}. Further assume that Eq. \eqref{eq:cond_popsize} is satisfied for a positive real function $\nu$ and that $c_N=f(N)+o_{\sum}(c_N)$ for a function $f(x)=cx^{-\gamma}$ for $\gamma>0$ or $f(x)=c\log(x)$ for a constant $c>0$. Then, the convergence can be equivalently expressed as $(\mathcal{R}_{c_N^{-1}t})_{t\geq 0}\stackrel{d}{\to}(\Pi_{\mathcal{G}(t)})_{t\geq 0}$ in the Skorohod-sense, where $\mathcal{G}(t)=\int^t_0 (\nu(s))^{-\gamma} ds$, where $\gamma=0$ is used if $f(x)=c\log(x)$.          
\end{lemma}
\begin{proof}
$\mathcal{G}$ is the pseudo-inverse of $\lim_{N\to\infty}\mathcal{G}^{-1}_Nc_N$, so   
\begin{equation}\label{eq:compute_mathcalG}
\mathcal{G}(t)=\lim_{N\to\infty} F_N(tc_{N}^{-1}),
\end{equation}
since for a sequence of functions, the inverses converges iff the original functions converge and since $cf$  has the inverse $t\mapsto f^{-1}(c^{-1}t)$. The shift by -1 does not alter the limit here, since its effect vanishes for $N\to\infty$ due to the multiplication with $c_N$. It is important to note here that any terms of order $o_{\sum}(c_N)$ can be omitted when computing $F_N$. Thus, we can replace $c_{N,r}$ by $c_{N_r}$ and even by $c^*_{N_r}=f(N_r)$ for a constant $c_2$. Since $N^{-1}N_{\lfloor tc_N^{-1}\rfloor}\to \nu(t)$, analoguous to \cite{griffiths1994sampling}, we can show, for $f(x)=cx^{-\gamma}$, 
\begin{align*}
\mathcal{G}(t)=&\lim_{N\to\infty}\sum_{r=1}^{[tc_{N}^{-1}]} \frac{c^*_{N_r}}{c^*_N}c^*_N=\lim_{N\to\infty} \sum_{r=1}^{[tc_{N}^{-1}]} \left(\frac{N}{N_r}\right)^{\gamma}c^*_N\\=&\lim_{N\to\infty}\int_0^t \sum_{r=1}^{[tc_N^{-1}]} \left(\frac{N_r}{N}\right)^{-\gamma}1_{[rc_N,(r+1)c_N)}(s) ds = \int_0^t (\nu(s))^{-\gamma} ds, 
\end{align*}
where for convergence, observe that there is pointwise convergence 

$$ \sum_{r=1}^{[tc_N^{-1}]} \left(\frac{N_r}{N}\right)^{-\gamma}1_{[rc_N,(r+1)c_N)}(s)= \left(\frac{N_{\lfloor sc_N^{-1}\rfloor}}{N}\right)^{-\gamma}\to (v(s))^{-\gamma}$$
for $s\in[0,t]$ inside the integral and that bounded convergence is applicable since the integrand is in $[0,1]$. If $f(x)$ is a logarithm, we have, using $k_r$ defined by $N_r=Nk_r$ for $0\leq r \leq c_{N}^{-1}t$  with $c_1(t)\leq k_r\leq c_2(t)$, 

\begin{align*}
\mathcal{G}(t)=&\lim_{N\to\infty}\sum_{r=1}^{[tc_{N}^{-1}]}c\log(N)^{-1}\frac{\log(N)}{\log(N_r)}\\=&\lim_{N\to\infty}c\sum_{r=1}^{[tc\log(N)]}\log(N)^{-1}\left(1-\frac{\log(k_r)}{\log(N)+\log(k_r)}\right)\\
=& t-\lim_{N\to\infty}c\sum_{r=1}^{[tc\log(N)]}\frac{\log(k_r)}{(\log(N)+\log(k_r))\log(N)}=t=\int_0^t (\nu(s))^0 ds.
\end{align*}
 \end{proof}

\begin{remark}\ 
\begin{itemize}
\item The integral representation of the time change is a deterministic version of the coalescent intensity from \citet[Sect. 1.3]{Kaj2003}, just applied to Cannings models leading to non-Kingman $\Lambda$-$n$-coalescents.
\item  As described in \citet[Section 4]{moehle_robust1998}, the time-changed $\Lambda$-$n$-coalescent limit can also be expressed by its infinitesimal rates 
$$\lambda^{(\nu)}_{n,k}=(\nu(s))^{-\gamma}\int^1_0 x^{k-2}(1-x)^{n-k}\Lambda(dx)$$ for a merger of $k$ of $n$ present lineages. This is also the form in which the limit process of the diploid umbrella model from \cite{koskela2018robust} is given.
\item Conditioned that the limit coalescent $(\Pi_{\mathcal{G}(t)})_{t\geq 0}$ has at time $T_0=t_0$ coalesced into a state with $b$ blocks, what is the distribution of the waiting time $T$ for the next coalescence event? If $T=t$, this means that in the non-rescaled $\Lambda$-$n$-coalescent $(\Pi_t)_{t\geq 0}$, we wait $\mathcal{G}(t)-\mathcal{G}(t_0)$ for the next coalescence. This waiting time $T'$ in $(\Pi_t)_{t\geq 0}$  is exponentially distributed with parameter $\lambda_n$ (total rate of coalescence). Thus, 
\begin{equation}\label{eq:cond_coaldist}
P(T>t_0+t|T_0=t_0)=P(T'>\mathcal{G}(t_0+t)-\mathcal{G}(t_0))=e^{\lambda_b(\mathcal{G}(t_0+t)-\mathcal{G}(t_0))}.
\end{equation}
\end{itemize}
\end{remark}
\begin{proof}(of Corollary \ref{cor:gompertz})
The form of $\mathcal{G}$ is a direct consequence of Lemma \ref{lem:shifttc}, since $\nu(t)=exp(-\rho t)$. Then, plugging $\mathcal{G}$ into Eq. \eqref{eq:cond_coaldist} yields the distribution for the next coalescence event, the Gompertz distribution parameters as e.g. described in \citet[Eq. 3]{lenart2014moments} can be read off.
\end{proof}

\subsection{Proofs of convergence to a time-changed coalescent - modified Moran models}
The modified Moran models used in Theorems \ref{thm:modMoran2tcMMCthinned} and \ref{thm:modMoran2tcMMC} were introduced in \citet[Prop. 4]{huillet2013extended}, the latter model with a small modification to ensure that there is always a parent with at least two offspring, see also \citet[Example 4.1]{huillet2013extended}.
\begin{proof} (of Theorem \ref{thm:modMoran2tcMMCthinned}) Assume that $E((U'_N)_2)(N-1)^{-1}\nrightarrow\nrightarrow 0$ for $N\to\infty$ also holds for any subsequence. If not, restrict to a subsequence for which this is true and define the limit only along this subsequence.\\
First, we verify that $c_N=N^{-\gamma}+o_{\sum}(N^{-\gamma})$, thus converges to 0 and that the fixed-$N$ model converges to the $\Lambda$-$n$-coalescent. Let $c'_N$ be the coalescence probability in a fixed-$N$ modified Moran model with $U'_N$ as the number of offspring of the multiplying parent. This ensures $2((N)_2)^{-1} \leq c'_N\leq 1$.  Moreover, the assumptions made ensure  that $(Nc_N')_{N\in\N}$ has a lower bound $>0$, so we can define $A_n$ s.t. $P(A_N)c'_N=N^{-\gamma}$ for any $\gamma\in(1,2)$. Then, the following is satisfied for $N\to\infty$ and $X\stackrel{d}{=}\Lambda$
\begin{align}\label{eq:thin3M}
&c_N=c'_NP(A_N)+(1-P(A_N))\frac{2}{N(N-1)}=N^{-\gamma}+o(N^{-\gamma}),
\\&
\frac{E((U_N)_k)}{(N)_k c_N}\stackrel{\eqref{eq:3M_factmom}}{=}\frac{E((U'_N)_k)N^{-\gamma}\lambda_N}{(N)_kc_N}
\stackrel{\eqref{eq:3M_factmom}}{=} \frac{E(X^k)N^{-\gamma}}{c_N} \to E(X^{k-2}) \mbox{ for } k>3\nonumber
\end{align} 
This establishes the convergence to the $\Lambda$-$n$-coalescent in the fixed-$N$ case. Now we assume variable population sizes $(N_r)_{r\in\N}$. First, observe that, since $c_N=O(N^{-\gamma})$, it is enough to add occasionally a single individual from generation $r$ to $r-1$ to generate any population size changes allowed in Eq. \eqref{eq:cond_popsize} including bottlenecks which are instantaneous on the coalescent time scale. This single individual can then be added as offspring of a non-multiplying parent from the fixed-$N_r$ model or as an offspring of the already multiplying parent. To see the latter, observe that $E((U_N)_2)=c_N(N)_2\sim N^{-\gamma}N^2\to \infty$. Then, as in \citet[third remark p. 8]{huillet2013extended}, one has 
$$\frac{E(U_N)}{E((U_N)_2)}\sim \frac{E(U_N)}{E(U_N^2)}\leq \frac{1}{E(U_N)}.$$ 
If both $E((U_N)_2),E(U_N)\to \infty$, the equation above shows that $\frac{E(U_N)}{E((U_N)_2)}\to 0$. If  $E(U_N)\nrightarrow \infty$ but $E((U_N)_2)$ does, we still have $\frac{E(U_N)}{E((U_N)_2)}\to 0$ for $N\to\infty$. Thus, Proposition \ref{prop:Moranlimit} allows to add the one individual also to the already multiplying parent.\\ Thus, we have verified all conditions but Eq. \eqref{eq:cond_speed_cN} to apply Proposition \ref{lem:3Mconv}. However, this follows from $c_N$ regularly varying. Since $c_N=N^{-\gamma}+o(N^{-\gamma})$, we can also shift the non-linear time-change to the limit due to Lemma \ref{lem:shifttc}. 
\end{proof}
For the proof of the next theorem, we use the following 
\begin{lemma}\label{lem:condspeed_modM}
Let $\Lambda\in\mathcal{M}[0,1]$ and let $U'_N$ be distributed as in Eq. \eqref{eq:LambdamodMoran} for any $N\in\mathbb{N}$. Let $(N_r)_{r\in\N}$ satisfy Eq. \eqref{eq:cond_popsize}. Then, Eq. \eqref{eq:cond_speed_cN} is satisfied. 
\end{lemma} 
\begin{proof}
Eq. \ref{eq:3M_factmom} shows $c_N=\lambda_N^{-1}$, where $\lambda_N$ is the total transition rate for the first jump of a $\Lambda$-$N$-coalescent. Without restriction, assume $N_r\geq N$ and that the $N$-coalescent is just the restriction of the $N_r$-coalescent on individuals $\{1,\ldots,N\}$. Any merger in the $N$-coalescent is then also a merger in the $N_r$-coalescent, which shows $\lambda_N\leq \lambda_{N_r}$. In contrast, the first merger in the $N_r$-coalescent is only a merger in the $N$-coalescent if it features at least two individuals from $\{1,\ldots,N\}$. The probability of this is bounded from below by the probability $\frac{N(N-1)}{N_r(N_r-1)}$ that the first two of the blocks merged in the $N$ are from $\{1,\ldots,N\}$. This implies
$$
0.5(c_2(t))^{-2}\stackrel{\eqref{eq:cond_popsize}}{<}\frac{N(N-1)}{N_r(N_r-1)}\leq \frac{c_{N}}{c_{N_r}}= \frac{\lambda_N}{\lambda_{N_r}} \leq 1 
$$
\end{proof}

\begin{proof} (of Theorem \ref{thm:modMoran2tcMMC}) In the fixed-$N$ case, Eq. \eqref{eq:LambdamodMoran} implies that $E((U'_N)_2)(N-1)^{-1}\to 0$ necessarily needs that $\lambda_N\to\infty$ for $N\to\infty$. This is equivalent to $\int_0^1 x^{-2}\Lambda(dx)=\infty$, see \cite[Eq. 7]{Pitman1999}. Thus, convergence to the $\Lambda$-$n$-coalescent is shown in \cite[Prop. 3.4]{huillet2013extended}. Now, switch to variable population sizes $(N_r)_{r\in\N}$. Since $Nc_N=E((U_N)_2)(N-1)^{-1}\to 0$ for $N\to\infty$, it is enough to add one individual per generation to cover any population growth profile covered by Eq. \eqref{eq:cond_popsize}. This can always be done by letting a parent not reproducing in the fixed-size model reproduce (once).  To add as further offspring of the multiplying parent, assume $E((U_N)_2)\to\infty$ for $N\to\infty$. Then, $\frac{E(U_N)}{E((U_N)_2)}\to 0$ as shown in the proof of Theorem \ref{thm:modMoran2tcMMCthinned}. Thus, Proposition \ref{prop:Moranlimit} provides that at most $A_{N,r}\leq c_4E(U_N)$ for arbitrary $c_4>0$ individuals can be added per generation to the already multiplying parent. This allows for adding up to any fixed number $k\in\mathbb{N}$ individuals per generation. Additionally, Eq. \eqref{eq:cond_speed_cN} is satisfied due to Lemma \ref{lem:condspeed_modM}.\\ 
We can thus apply Lemma \ref{lem:3Mconv}, with an arbitrary allocation of additional individuals that yields a modified Moran model.
\end{proof}

To prove the Corollaries \ref{cor:beta_thm2}, \ref{cor:betanonthin}, we collect some properties of the modified Moran models with $U_N=U'_N$ with $U'_N$ given by Eq. \eqref{eq:LambdamodMoran} leading to Beta-$(a,b)$-coalescents for $a\in(0,2]$, $b>0$. From \citet[Eq. (10)+Corollary A.1]{huillet2013extended}, 
\begin{equation}\label{eq:beta_modMoran}
c_N\sim \frac{(2-a)\Gamma(b)}{\Gamma(a+b)}N^{a-2} \mbox{ for }a<2.
\end{equation}

\begin{proof}(of Corollary \ref{cor:beta_thm2})
From Eq. \eqref{eq:beta_modMoran}, it follows that for $a\in(0,1)$, $\Lambda=Beta(a,b)$ satisfies $E((U'_N)_2)(N-1)^{-1}=Nc_N=O(N^{a-1}) \to 0$ for $N\to\infty$. Thus, such $\Lambda$-$n$-coalescents are covered by Theorem \ref{thm:modMoran2tcMMC}. Additionally from Eq. \eqref{eq:beta_modMoran}, $c_N$ has a form that is covered by Lemma \ref{lem:shifttc}, which allows to shift the time-change $\mathcal{G}$ in Theorem \ref{thm:modMoran2tcMMC} to the limit coalescent and also shows the form of $\mathcal{G}$ in Corollary \ref{cor:beta_thm2}.  
\end{proof}

\begin{proof}(of Corollary \ref{cor:betanonthin})
We reiterate the proof of Theorem \ref{thm:modMoran2tcMMC}. Let $\Lambda=Beta(a,b)$ for $a\in(1,2)$, which satisfies $\int x^{-2}\Lambda(dx)=\infty$. Thus, in the fixed-$N$ case, again  \cite[Prop. 3.4]{huillet2013extended} ensures the convergence of the discrete genealogies to the $\Lambda$-$n$-coalescent when properly rescaled for $N\to\infty$. Furthermore, Lemma \ref{lem:condspeed_modM} shows that Eq. \eqref{eq:cond_speed_cN} is satisfied  Since $\nu(t)=exp(-\rho t)$, we can use $N_r=\lfloor N(1-\rho c_N)\rfloor$ to satisfy \eqref{eq:cond_popsize}. Thus, we only need to show that the population size increase per generation does not violate the conditions of Proposition \ref{lem:3Mconv}. Indeed,
\begin{align*}
d_{N,r} &=\lfloor N( 1-c_N\rho)^{r}\rfloor -\lfloor N( 1-\rho c_N)^{r+1}\rfloor
\leq N\left( 1 - \left( 1-\rho c_N\right)\right)+1\\
&= N(\rho c_N)+1
\end{align*}
individuals at most have to be added. These can be added as $A_{n,r}$ additional offspring of the multiplying parent from the fixed-$N_r$ model, if the condition to apply it from Lemma \ref{lem:3Mconv} are met. For $\Lambda$ considered here, one has $E((U_N)_2))\sim N^2 c_N\sim N^{a}\to \infty$, see Eq. \eqref{eq:beta_modMoran}. From Lemma \ref{lem:3Mconv} we see that then we are allowed to add $O(E(U_N))$ individuals. \citet[Remark p.9]{huillet2013extended} shows $E(U_N)=c_5Nc_N$ for a constant $c_5>0$, so such growth is indeed covered (and we can then still use $A_{N,r}<d_{N_r}$ and add the other individuals to non-reproducing parents from the fixed-$N_r$ model). Thus, we can establish convergence using Proposition \ref{lem:3Mconv} and shift the time-change $\mathcal{G}$ to the limit using Lemma \ref{lem:shifttc}, since $c_N$ is essentially a negative power of $N$. 
\end{proof}

\begin{proof}(of Proposition \ref{prop:Moranlimit})
For $\Lambda=\delta_0$, Eq. \eqref{eq:LambdamodMoran} shows $U'\equiv 2$, so the modified Moran model is the normal Model model in this case. Since $E((U'_N)_2)(N-1)^{-1}=2(N-1)^{-1}\to 0$ for $N\to\infty$, Theorem \ref{thm:modMoran2tcMMC} applies. Since $c_N=2(N(N-1))^{-1}=2N^{-2} + o(N^{-2})$, we can apply \ref{lem:shifttc} to shift the time-change $\mathcal{G}$ to the coalescent limit. 
\end{proof}
\subsection{Proofs of converging to a time-changed coalescent - model from \citet{Schweinsberg2003}}

\begin{proof} (of Lemma \ref{lem:popsizechange_schwein})
It suffices to reiterate the proof of \citet[Lemma 5]{Schweinsberg2003} briefly. For $u\in[0,1]$, consider the generating function $f(u):=E(u^X)$. Let $d'_{n,r}=d_{N,r}/N^r$. Then,  $S_{N,r}:=\sum_{i=1}^{N_r} X^{(r)}_i$ fulfills

$$P(S_{N,r}\leq N_{r-1})\leq u^{-N_r(1+d'_{N,r})}E(u^{S_{N,r}})=(u^{-(1+d'_{N,r})}f(u))^{N_r}.$$ Since $f(1)=1$ and $f'(1)=\mu>1$, there exists $u_0\in(0,1)$ and $\epsilon>0$ so that $u_0^{1+\epsilon}>f(u_0)$. Moreover, for this $\epsilon$ we find $N_0\in\N$ so that $d_N<\epsilon$ for $N\geq N_0$. For such $N$, as computed above, one gets 

$$P(S_{N,r}\leq N_{r-1})\leq (u_0^{-(1+d'_{N,r})}f(u_0))^{N_r}\leq A_1^{N_r} \leq A_1^{N^{-}(t)}=(\underbrace{A_1^{c^{-}(t)}}_{<1})^N,$$
where $A_1:=u_0^{-(1+\epsilon)}f(u_0)<1$. Setting $A:=A_1^{c^{-}(t)}$ completes the proof.
\end{proof}

\begin{proof}(of Theorem \ref{prop:schweinspopsized}) 
Recall that, for $1<\alpha<2$, the coalescence probability in this (fixed-$N$) model satisfies $c_N\sim C\alpha B(2-\alpha,\alpha)E(X)^{-\alpha}N^{1-\alpha}$, where $B$ is the Beta function, see \citet[Lemma 13]{Schweinsberg2003}. For $\alpha=1$, instead $c_N\sim (\log N)^{-1}$, see \citet[Lemma 16]{Schweinsberg2003}.
Check the conditions necessary to apply Lemma \ref{lem:cannpopsizetolambda}: The model and the assumptions above satisfy $c_N\to 0$ for $N\to\infty$, \eqref{eq:cond_popsize} and, since $c_N$ is regularly varying, also \eqref{eq:cond_speed_cN}. The changes of population sizes from generation to generation are enough to cover instantanous population size changes on the coalescent time scale (and these are the most extreme changes allowed in Eq. \eqref{eq:cond_popsize}): for a (coalescent time) instantaneous change of size $mN$, one can set $|d_{N,r}|=m\sqrt{c_N}=:d_N\to 0$ for $N\to\infty$ for $(\sqrt{c_N})^{-1}$ generations. Thus, only \eqref{eq:cond_phi} needs to be verified. Since $N_{r-1}$ offspring are sampled from $\sum_{i=1}^{N_r}X^{(r)}_{i}$ potential offspring, the transition probabilities of the discrete coalescent can be formulated analogously to Eq. \ref{eq:cann_transprob} as 

\begin{align*}
\Phi^{(N)}(r;a_1,\ldots,a_l)=&\frac{(N_r)_lE\left(\prod_{i=1}^l (X^{(r)}_i)_{a_i}\right)}{(N_{r-1})_{\sum^l_i a_i}}=\Phi^{(N_r)}(a_1,\ldots,a_l)\frac{(N_{r})_{\sum^l_i a_i}}{(N_{r-1})_{\sum^l_i a_i}}
\end{align*} 
This means one just needs to show that 

$$\Phi^{(N_r)}(a_1,\ldots,a_l)\left|\frac{\prod_{i=1}^l(N_r)_{a_i}}{\prod_{i=1}^l(N_{r-1})_{a_i}}-1\right|=o_{\sum}(c_N),$$ which follows from $c_N^{-1}\Phi^{(N_r)}(a_1,\ldots,a_l)\left|\frac{\prod_{i=1}^l(N_r)_{a_i}}{\prod_{i=1}^l(N_{r-1})_{a_i}}-1\right|\to 0$ uniformly in $r$. To show the latter, uniform convergence, proceed as following. First, observe that, since \eqref{eq:cond_speed_cN} holds, 

$$c_N^{-1}\Phi^{(N_r)}(a_1,\ldots,a_l)=\underbrace{c_{N_r}^{-1}\Phi^{(N_r)}(a_1,\ldots,a_l)}_{\to \phi(a_1,\ldots,a_l)}\frac{c_{N_r}}{c_N}$$ is uniformly bounded (again, since $N_r$ is bounded from below by $N^-(t)$, there is uniform convergence in $r$ of the first factor for $N\to\infty$). Thus, we only need to show $\left|\frac{\prod_{i=1}^l(N_r)_{a_i}}{\prod_{i=1}^l(N_{r-1})_{a_i}}-1\right|\to 0$. For this, observe that the function $x\mapsto\frac{a_1-x}{a_2-x}$ for $x<a_2$ is strictly increasing (decreasing) if $a_1-a_2>0$ (if $a_1-a_2<0$). This shows that there are $b_1,b_2\in\N_0$ so that 

$$\left(\frac{N_r-b_1}{N_{r-1}-b_1}\right)^{\sum_i^l a_i}\leq\frac{\prod_{i=1}^l(N_r)_{a_i}}{\prod_{i=1}^l(N_{r-1})_{a_i}}\leq\left(\frac{N_r-b_2}{N_{r-1}-b_2}\right)^{\sum_i^l a_i}.$$ This implies that it is sufficient to show $\left|\left(\frac{N_r-b}{N_{r-1}-b}\right)^{\sum_i^l a_i}-1\right|\to 0$ for $N\to\infty$ any $N^-(t)>b\geq 0$, which follows from $\left|\left(\frac{N_r-b}{N_{r-1}-b}\right)-1\right|\to 0$ uniformly in $r$. 
Further computation shows

$$
\left|\frac{N_r-b}{N_{r-1}-b}-1\right|=\frac{|d_{N,r}|}{N_{r-1}-b}\leq \frac{d_N}{N_{r-1}-b}.   
$$
Since $d_N\to 0$, this vanishes uniformly. Thus, Lemma \ref{lem:cannpopsizetolambda} can be applied, establishing convergence of  $(\tilde{\mathcal{R}}^{(N)}_{[\mathcal{G}_N^{-1}(t)]})_{t\geq 0}\stackrel{d}{\to}(\Pi_{t})_{t\geq 0}$. Lemma \ref{lem:shifttc} then ensures that the time-change $\mathcal{G}$ can be shifted to the limit, since $c_N$ is either essentially a negative power or a logarithm of $N$.
\end{proof}

%% file: mmc_popsized.bbl
\begin{thebibliography}{37}
\providecommand{\natexlab}[1]{#1}
\providecommand{\url}[1]{\texttt{#1}}
\expandafter\ifx\csname urlstyle\endcsname\relax
  \providecommand{\doi}[1]{doi: #1}\else
  \providecommand{\doi}{doi: \begingroup \urlstyle{rm}\Url}\fi

\bibitem[Alter and Louzoun(2016)]{alter2016population}
I.~Alter and Y.~Louzoun.
\newblock Population growth combined with wide offspring distributions can
  increase fixation rate and reduce genetic diversity.
\newblock \emph{Bulletin of mathematical biology}, 78\penalty0 (7):\penalty0
  1477--1492, 2016.

\bibitem[Birkner et~al.(2009)Birkner, Blath, M{\"o}hle, Steinr{\"u}cken, and
  Tams]{Birkner2009}
M.~Birkner, J.~Blath, M.~M{\"o}hle, M.~Steinr{\"u}cken, and J.~Tams.
\newblock A modified lookdown construction for the {X}i-fleming-viot process
  with mutation and populations with recurrent bottlenecks.
\newblock \emph{Alea}, 6:\penalty0 25--61, 2009.

\bibitem[Cannings(1974)]{cannings1974latent}
C.~Cannings.
\newblock The latent roots of certain markov chains arising in genetics: a new
  approach, i. haploid models.
\newblock \emph{Advances in Applied Probability}, 6\penalty0 (2):\penalty0
  260--290, 1974.

\bibitem[Cannings(1975)]{Cannings1975}
C.~Cannings.
\newblock The latent roots of certain markov chains arising in genetics: a new
  approach, ii. further haploid models.
\newblock \emph{Advances in Applied Probability}, pages 264--282, 1975.

\bibitem[Desai et~al.(2013)Desai, Walczak, and Fisher]{Desai2013}
M.~M. Desai, A.~M. Walczak, and D.~S. Fisher.
\newblock Genetic diversity and the structure of genealogies in rapidly
  adapting populations.
\newblock \emph{Genetics}, 193\penalty0 (2):\penalty0 565--585, 2013.

\bibitem[Donnelly and Kurtz(1999)]{donnelly1999particle}
P.~Donnelly and T.~G. Kurtz.
\newblock Particle representations for measure-valued population models.
\newblock \emph{The Annals of Probability}, 27\penalty0 (1):\penalty0 166--205,
  1999.

\bibitem[Eldon(2012)]{eldon2012age}
B.~Eldon.
\newblock Age of an allele and gene genealogies of nested subsamples for
  populations admitting large offspring numbers.
\newblock \emph{arXiv preprint arXiv:1212.1792}, 2012.

\bibitem[Eldon and Wakeley(2006)]{Eldon2006}
B.~Eldon and J.~Wakeley.
\newblock Coalescent processes when the distribution of offspring number among
  individuals is highly skewed.
\newblock \emph{Genetics}, 172\penalty0 (4):\penalty0 2621--2633, 2006.

\bibitem[Eldon et~al.(2016)Eldon, Riquet, Yearsley, Jollivet, and
  Broquet]{eldon2016current}
B.~Eldon, F.~Riquet, J.~Yearsley, D.~Jollivet, and T.~Broquet.
\newblock Current hypotheses to explain genetic chaos under the sea.
\newblock \emph{Current zoology}, 62\penalty0 (6):\penalty0 551--566, 2016.

\bibitem[Griffiths and Tavare(1994)]{griffiths1994sampling}
R.~C. Griffiths and S.~Tavare.
\newblock Sampling theory for neutral alleles in a varying environment.
\newblock \emph{Philosophical transactions: biological sciences}, pages
  403--410, 1994.

\bibitem[Hartmann and Huillet(2018)]{hartmann2018large}
A.~K. Hartmann and T.~Huillet.
\newblock Large-deviation properties of the extended moran model.
\newblock \emph{Physical Review E}, 98\penalty0 (4):\penalty0 042416, 2018.

\bibitem[Hedgecock and Pudovkin(2011)]{Hedgecock2011}
D.~Hedgecock and A.~I. Pudovkin.
\newblock Sweepstakes reproductive success in highly fecund marine fish and
  shellfish: a review and commentary.
\newblock \emph{Bulletin of Marine Science}, 87\penalty0 (4):\penalty0
  971--1002, 2011.

\bibitem[Hoscheit and Pybus(2018)]{hoscheit2018multifurcating}
P.~Hoscheit and O.~Pybus.
\newblock The multifurcating skyline plot.
\newblock \emph{bioRxiv}, page 356097, 2018.

\bibitem[Hudson(2002)]{Hudson2002}
R.~R. Hudson.
\newblock Generating samples under a wright--fisher neutral model of genetic
  variation.
\newblock \emph{Bioinformatics}, 18\penalty0 (2):\penalty0 337--338, 2002.

\bibitem[Huillet and M{\"o}hle(2013)]{huillet2013extended}
T.~Huillet and M.~M{\"o}hle.
\newblock On the extended moran model and its relation to coalescents with
  multiple collisions.
\newblock \emph{Theoretical population biology}, 87:\penalty0 5--14, 2013.

\bibitem[Irwin et~al.(2016)Irwin, Laurent, Matuszewski, Vuilleumier, Ormond,
  Shim, Bank, and Jensen]{irwin2016importance}
K.~K. Irwin, S.~Laurent, S.~Matuszewski, S.~Vuilleumier, L.~Ormond, H.~Shim,
  C.~Bank, and J.~D. Jensen.
\newblock On the importance of skewed offspring distributions and background
  selection in virus population genetics.
\newblock \emph{Heredity}, 2016.

\bibitem[Kaj and Krone(2003)]{Kaj2003}
I.~Kaj and S.~M. Krone.
\newblock The coalescent process in a population with stochastically varying
  size.
\newblock \emph{Journal of applied probability}, 40\penalty0 (1):\penalty0
  33--48, 2003.

\bibitem[Kato et~al.(2017)Kato, Vasco, Sugino, Narushima, and
  Krasnitz]{Kato171060}
M.~Kato, D.~A. Vasco, R.~Sugino, D.~Narushima, and A.~Krasnitz.
\newblock Sweepstake evolution revealed by population-genetic analysis of
  copy-number alterations in single genomes of breast cancer.
\newblock \emph{Royal Society Open Science}, 4\penalty0 (9), 2017.
\newblock \doi{10.1098/rsos.171060}.
\newblock URL \url{http://rsos.royalsocietypublishing.org/content/4/9/171060}.

\bibitem[Kelleher et~al.(2016)Kelleher, Etheridge, and McVean]{msprime}
J.~Kelleher, A.~M. Etheridge, and G.~McVean.
\newblock Efficient coalescent simulation and genealogical analysis for large
  sample sizes.
\newblock \emph{PLoS Comput Biol}, 12\penalty0 (5):\penalty0 1--22, 05 2016.
\newblock \doi{10.1371/journal.pcbi.1004842}.
\newblock URL \url{http://dx.doi.org/10.1371%2Fjournal.pcbi.1004842}.

\bibitem[Koskela and {Wilke Berenguer}(2019)]{koskela2018robust}
J.~Koskela and M.~{Wilke Berenguer}.
\newblock Robust model selection between population growth and multiple merger
  coalescents.
\newblock \emph{Mathematical biosciences}, 311:\penalty0 1--12, 2019.

\bibitem[Lenart(2014)]{lenart2014moments}
A.~Lenart.
\newblock The moments of the {G}ompertz distribution and maximum likelihood
  estimation of its parameters.
\newblock \emph{Scandinavian Actuarial Journal}, 2014\penalty0 (3):\penalty0
  255--277, 2014.

\bibitem[Li and Durbin(2011)]{li2011inference}
H.~Li and R.~Durbin.
\newblock Inference of human population history from individual whole-genome
  sequences.
\newblock \emph{Nature}, 475\penalty0 (7357):\penalty0 493--496, 2011.

\bibitem[Matuszewski et~al.(2017)Matuszewski, Hildebrandt, Achaz, and
  Jensen]{Matuszewski2017}
S.~Matuszewski, M.~E. Hildebrandt, G.~Achaz, and J.~D. Jensen.
\newblock Coalescent processes with skewed offspring distributions and
  non-equilibrium demography.
\newblock \emph{Genetics}, 2017.
\newblock ISSN 0016-6731.
\newblock \doi{10.1534/genetics.117.300499}.
\newblock URL
  \url{http://www.genetics.org/content/early/2017/11/10/genetics.117.300499}.

\bibitem[M{\"o}hle(1998)]{moehle_robust1998}
M.~M{\"o}hle.
\newblock Robustness results for the coalescent.
\newblock \emph{Journal of Applied Probability}, 35\penalty0 (2):\penalty0
  438--447, 1998.
\newblock ISSN 00219002.
\newblock URL \url{http://www.jstor.org/stable/3215697}.

\bibitem[M{\"o}hle(2002)]{mohle2002coalescent}
M.~M{\"o}hle.
\newblock The coalescent in population models with time-inhomogeneous
  environment.
\newblock \emph{Stochastic processes and their applications}, 97\penalty0
  (2):\penalty0 199--227, 2002.

\bibitem[M{\"o}hle and Sagitov(2001)]{Moehle2001}
M.~M{\"o}hle and S.~Sagitov.
\newblock A classification of coalescent processes for haploid exchangeable
  population models.
\newblock \emph{The Annals of Probability}, 29\penalty0 (4):\penalty0
  1547--1562, 2001.

\bibitem[Neher and Hallatschek(2013)]{Neher2013a}
R.~A. Neher and O.~Hallatschek.
\newblock Genealogies of rapidly adapting populations.
\newblock \emph{Proceedings of the National Academy of Sciences}, 110\penalty0
  (2):\penalty0 437--442, 2013.

\bibitem[Pitman(1999)]{Pitman1999}
J.~Pitman.
\newblock Coalescents with multiple collisions.
\newblock \emph{Annals of Probability}, pages 1870--1902, 1999.

\bibitem[Polanski et~al.(2003)Polanski, Bobrowski, and
  Kimmel]{polanski2003note}
A.~Polanski, A.~Bobrowski, and M.~Kimmel.
\newblock A note on distributions of times to coalescence, under time-dependent
  population size.
\newblock \emph{Theoretical population biology}, 63\penalty0 (1):\penalty0
  33--40, 2003.

\bibitem[Sagitov(1999)]{Sagitov1999}
S.~Sagitov.
\newblock The general coalescent with asynchronous mergers of ancestral lines.
\newblock \emph{Journal of Applied Probability}, 36\penalty0 (4):\penalty0
  1116--1125, 1999.

\bibitem[Schweinsberg(2003)]{Schweinsberg2003}
J.~Schweinsberg.
\newblock Coalescent processes obtained from supercritical {G}alton--{W}atson
  processes.
\newblock \emph{Stochastic processes and their applications}, 106\penalty0
  (1):\penalty0 107--139, 2003.

\bibitem[Schweinsberg(2017)]{Schweinsberg2017}
J.~Schweinsberg.
\newblock Rigorous results for a population model with selection ii: genealogy
  of the population.
\newblock \emph{Electronic Journal of Probability}, 22, 2017.

\bibitem[Slatkin and Hudson(1991)]{slatkin1991pairwise}
M.~Slatkin and R.~R. Hudson.
\newblock Pairwise comparisons of mitochondrial dna sequences in stable and
  exponentially growing populations.
\newblock \emph{Genetics}, 129\penalty0 (2):\penalty0 555--562, 1991.

\bibitem[Spence et~al.(2016)Spence, Kamm, and Song]{Spence1549}
J.~P. Spence, J.~A. Kamm, and Y.~S. Song.
\newblock The site frequency spectrum for general coalescents.
\newblock \emph{Genetics}, 202\penalty0 (4):\penalty0 1549--1561, 2016.
\newblock ISSN 0016-6731.
\newblock \doi{10.1534/genetics.115.184101}.
\newblock URL \url{http://www.genetics.org/content/202/4/1549}.

\bibitem[Steinr{\"u}cken et~al.(2013)Steinr{\"u}cken, Birkner, and
  Blath]{Steinruecken2013}
M.~Steinr{\"u}cken, M.~Birkner, and J.~Blath.
\newblock Analysis of dna sequence variation within marine species using
  beta-coalescents.
\newblock \emph{Theoretical population biology}, 87:\penalty0 15--24, 2013.

\bibitem[Tellier and Lemaire(2014)]{Tellier2014}
A.~Tellier and C.~Lemaire.
\newblock Coalescence 2.0: a multiple branching of recent theoretical
  developments and their applications.
\newblock \emph{Molecular ecology}, 23\penalty0 (11):\penalty0 2637--2652,
  2014.

\bibitem[Terhorst et~al.(2017)Terhorst, Kamm, and Song]{terhorst2017robust}
J.~Terhorst, J.~A. Kamm, and Y.~S. Song.
\newblock Robust and scalable inference of population history from hundreds of
  unphased whole genomes.
\newblock \emph{Nature genetics}, 49\penalty0 (2):\penalty0 303, 2017.

\end{thebibliography}
